\documentclass[reqno, a4paper, 12pt]{amsart}
\usepackage[english]{babel}
\usepackage{mathptmx}
\usepackage{hyperref}
\usepackage{graphicx}
\hypersetup{unicode=true, pdfauthor=Phung Van Manh, pdfstartview=FitH}
\usepackage{fullpage}
\newtheorem{theorem}{Theorem}[section]
\newtheorem{corollary}[theorem]{Corollary}
\newtheorem{lemma}[theorem]{Lemma}

\theoremstyle{definition}
\newtheorem{definition}[theorem]{Definition}

\theoremstyle{remark}
\newtheorem{remark}[theorem]{Remark}

\numberwithin{equation}{section}

\newcommand{\ZZ}{\mathbb{Z}} 
\newcommand{\NN}{\mathbb{N}} 
\newcommand{\RR}{\mathbb{R}} 

\newcommand{\LL}{\mathcal{L}}

\begin{document}
\title{On the convergence of Kergin and Hakopian interpolants at Leja sequences for the disk }
\author{Phung Van Manh}
\address{Institut de Math\'{e}matiques, Universit\'{e} de Toulouse III, 31062, Toulouse Cedex 9, France and Department of Mathematics, Hanoi University of Education,
136 Xuan Thuy street, Cau Giay, Hanoi, Vietnam}
\email{manhlth@gmail.com}





\begin{abstract} We prove that Kergin interpolation polynomials and Hakopian interpolation polynomials at the points of a Leja sequence for the unit disk $D$ of a sufficiently smooth function $f$ in a neighbourhood of $D$ converge uniformly to $f$ on $D$. Moreover,  when $f\in C^\infty(D)$, all the derivatives of the interpolation polynomials converge uniformly to the corresponding derivatives of $f$.
\end{abstract}

\maketitle
   
\section{Introduction} 
Kergin and Hakopian interpolants were introduced independently about thirty years ago as natural multivariate generalizations of univariate Lagrange interpolation. The construction of these interpolation polynomials requires the use of points, usually called nodes, with which one obtains a number of natural mean value linear forms which provide the interpolation conditions. Kergin interpolation polynomials also interpolate  in the usual sense, that is, the interpolation polynomial and the interpolated function coincide on the set of nodes but this condition no longer characterizes it. The general definition is recalled below. Approximation properties of Kergin and Hakopian interpolation polynomials have been deeply investigated, see e.g., \cite{anderssonpassare, bloom, bloomcalvi, bloomcalvi2, filipsson}. Elegant results were in particular obtained in the two-dimensional case when the nodes forms a complete set of roots of unity (viewed as a subset of $\RR^2$). Thus, in \cite{liang}, Liang established a formula for Hakopian interpolation at the roots of unity in $\RR^2$, and later together with L\"{u} \cite{liang.lu} estimated the remainder and proved that Hakopian interpolation polynomials at the roots of unity of a function of class $C^2$ in a neighbourhood of the closed unit disk $D\subset \RR^2$ converge uniformly to the function on $D$. Thanks to Liang's formula, further authors investigated (weighted) mean convergence of Hakopian interpolation (see \cite{liang.lu, liang.feng.sun}). On other hand, in 1997, using a beautiful formula for Kergin interpolation at nodes in general position in $\RR^2$, Bos and Calvi \cite{boscalvi1} independently established a similar convergence result for Kergin interpolation. If $C_n$ denotes the set of $n$-th roots of unity, $\mathcal{H}[C_n; \cdot]$ (resp. $\mathcal{K}[C_n; \cdot]$) 
the Hakopian (resp. Kergin) projector, the results can be stated as
$$\mathcal{K}[C_n; f] \rightarrow f\; \textrm{and}\; \mathcal{H}[C_n; f] \rightarrow f, \quad\textrm{uniformly on $D$,\; for every $f\in C^2(D)$}.$$
In the above results, going from $n$ to $n+1$, we need to change all the nodes and it seems natural to look for similar results in which $C_n$ would be replaced by a set $E_n$ such that $E_{n}\subset E_{n+1}$, which comes to find \emph{sequences} of nodes rather than sequences of arrays of nodes. It is the purpose of this note to exhibit such sequences. They enable us to obtain series expansions of the form
$$f(x)=\sum_{d=0}^\infty \mu\left(e_0,\dots,e_d, \textup{D}^df(\cdot)(x-e_0,\dots,x-e_{d-1})\right)$$
for $f\in C^\infty(D)$, where the $\mu(e_0,\dots,e_d,\cdot)$ are 
certain mean value linear forms (whose definition will be specified below) and $\textup{D}^df(a)$ denotes the $d$-th total derivative of $f$. The sequences that we shall use are Leja sequences for $D$ and the results of the present paper have been made possible by recent  progresses on the study of Leja sequences (and associated constants) contained in \cite{biacal, jpcpvm, calvimanh2}. To treat the case of Hakopian interpolation, we shall prove a formula for Hakopian interpolation at nodes in general position in $\RR^2$ which reduces to Liang's formula when the nodes form a complete set of roots of unity and which is of independent interest. The proof of our convergence results requires a somewhat higher level of smoothness than in the case of interpolation at the roots of unity. The question whether we can weaken the smoothness of the interpolated function is still unanswered. Moreover, when the interpolated function is in the class $C^\infty$, we show that all the derivatives of the interpolation polynomials converge uniformly to the corresponding derivatives of the interpolated function. 
 
\emph{Notations}. The scalar product of $x=(x^1,\ldots,x^N)$ and $y=(y^1,\ldots,y^N)$ in $\RR^N$ is defined by  
$\langle x, y\rangle:=\sum_{j=1}^Nx^jy^j,$
and the corresponding norm of $x$ is $\|x\|=\sqrt{\langle x,x\rangle}$. Let $K$ be a compact set in $\RR^N$. For each continuous function $f$ on $K$ we set $\|f\|_{K}=\sup\{|f(x)|: x\in K\}$. The space of $k$-times continuously differentiable functions on a neighbourhood of $K$ is denoted by  $C^k(K)$. For $f\in C^k(K)$, $k\geq 1$, we set 
$$\textup{D}^{\alpha}f=\frac{\partial ^{|\alpha|}f}{(\partial x^1)^{\alpha_1}\cdots (\partial x^N)^{\alpha_N} },\quad \alpha=(\alpha_1,\ldots,\alpha_N),\,\,\, |\alpha|=\alpha_1+\cdots+\alpha_N\leq k,$$
$$\textup{D}_yf(x)=\textup{D}f(x)(y)=\sum_{j=1}^N\frac{\partial f}{\partial x^j}(x)y^j,\quad x\in K,\,\,\, y=(y^1,\ldots,y^N).$$  
The euclidean norm of the linear form $\textup{D}f(x)$ is denoted by $\|\textup{D}f(x)\|$. We have
$$\|\textup{D}f(x)\|=\Big[\sum_{j=1}^N\big(\frac{\partial f}{\partial x^j}(x)\big)^2\Big ]^{\frac 12}.$$ 
We also denote by $\mathcal P_d(\RR^N)$ the space of polynomials of $N$ variables and degree at most $d$.  

\section{The definition of Kergin and Hakopian interpolants}
It is convenient to recall some definitions and properties of interpolation polynomials in their full generality. In particular, we shall introduce Kergin and Hakopian interpolation polynomials as particular cases of a more general procedure. 
 
Given a convex subset $\Omega\subset \RR^N$ and a tuple $A$ of $d+1$ not necessarily distinct points in $\Omega$,  $A=(a_0,a_1, \dots, a_d)\in \Omega^{d+1}$, the simplex functional $\int\limits_{[a_0,\dots,\, a_d]}$ is defined on the space of 
continuous functions $C(\Omega)$ by the relation
\begin{equation}\int\limits_{[a_0,\dots,\, a_d]}f:=\int\limits_{\Delta_d} f\Big(a_0+\sum_{j=1}^d t_j(a_j-a_0)\Big) \textup{d}t,\quad f\in C(\Omega),\,\,\, d\geq 1,\end{equation}
where $\textup{d}t=\textup{d}t_1\cdots \text {d}t_d$ stands for the ordinary Lebesgue measure on the standard simplex $\Delta_d=\{(t_1,t_2,\dots, t_d)\in [0,\,1]^d,\; \sum_{j=1}^d t_j \leq 1\}$. In the case $d=0$ we set $\int\limits_{[a_0]}f=f(a_0)$.
\par 
The following theorem leads us to the definition of mean-value interpolation. Its proof can be found in \cite{goodman} or \cite{filipsson}.  

\begin{theorem}
Let $\Omega$ be an open convex subset of $\RR^N$, $A=(a_0,\dots,a_d)$ be a tuple in $\Omega$ and let $k\in\{0,\dots,d\}$. For every function $f\in C^{d-k}(\Omega)$, the set of all $(d-k)$-times continuously differentiable functions on $\Omega$, there exists a unique polynomial $P$ on $\RR^N$ of degree at most $d-k$ such that   
\begin{equation}\label{eq:intercondMV}\int\limits_{[a_0,\dots,a_{j+k}]}\textup{D}^\alpha(f-P)=0,\quad |\alpha|=j,\quad j=0,\dots, d-k. \end{equation}
\begin{definition}
The polynomial $P$ in (\ref{eq:intercondMV}) is called the \emph{$k$-th mean-value interpolation polynomial} of $f$ at $A$ and is denoted by $\LL^{(k)}[A;f]$ or $\LL^{(k)}[a_0,\dots,a_d;f]$. 
\end{definition}
\end{theorem}
There is an explicit but rather complicated formula for mean-value interpolation, see \cite[Theorem 1]{goodman} or \cite[Theorem 4.3]{filipsson}. Here we summarize a few basic properties of mean-value interpolation.
\begin{enumerate}
\item The polynomial $\LL^{(k)}[A;f]$ does not depend on the ordering of the points in $A$,
\item  The operator $ \LL^{(k)}[A] :\; f\in C^{d-k}(\Omega)\longmapsto \LL^{(k)}[A;f]\in \mathcal{P}_{d-k}(\RR^N)$
is a continuous linear projector (when $C^{d-k}(\Omega)$ is equipped with its standard topology),
\item For each $f\in C^{d-k}(\Omega)$, the map $A\in \Omega^{d+1}\longmapsto \LL^{(k)}[A;f]\in \mathcal{P}_{d-k}(\RR^N)$ is continuous,
\item For any affine mapping $\Psi: \RR^N\to\RR^M$ and any suitably defined function $f$, we have $ \LL^{(k)}[A;f\circ \Psi]= \LL^{(k)}[\Psi(A);f]\circ\Psi$ 
\item The polynomial $\LL^{(0)}[A;f]$ interpolates $f$ at the $a_j$'s and becomes Taylor polynomial of $f$ at $a$ of order $d$ when $a_0=\cdots=a_d=a$. 
 \end{enumerate}

The most interesting mean-value interpolation polynomials are Kergin interpolants which correspond to the case $k=0$,
\begin{equation} \mathcal{K}[a_0,\dots , a_d;f]=\LL^{(0)}[a_0,\dots,a_d;f], \end{equation} and Hakopian interpolants which correspond to the case $k=N-1$ and $d\geq N-1$, 
\begin{equation} \mathcal{H}[a_0,\dots , a_{d};f]=\LL^{(N-1)}[a_0,\dots,a_{d};f]. \end{equation}   
   When the $a_j$'s are in general position in $\RR^N$- that is, every subset of $N+1$ points of $A$ defines an affine basis of $\RR^N$- then Kergin operator extends to functions of class $C^{N-1}$, see \cite[p. 206-207]{boscalvi1}. On the other hand, under the same condition on the points, Hakopian interpolation is characterized by the following relation. For $P\in\mathcal P_{d-N+1}(\RR^N)$,  
\begin{equation}\label{for:reduce-hakop} P=\mathcal{H}[a_0,\dots , a_{d};f] 
 \iff  \int\limits_{[a_{i_1}, \dots, a_{i_N}]}(f-P)=0,\quad 0\leq i_1 < \dots < i_{N} \leq d. \end{equation}  
Hence, in that case, derivatives are no longer involved and the Hakopian operator extends to continuous functions. 
\section{Error formulas for Hakopian and Kergin interpolants in $\RR^2$} 
We now restrict ourselves to the two-dimensional case.  
\subsection{} For $x=(x^1,x^2)\in \RR^2$, we denote by $x^{\bot}:=(-x^2,x^1)$, the image of $x$ under the rotation of center the origin and angle $\pi/2$. As usual, to $x=(x^1,x^2)$, we associate the complex number $x^1+ix^2$ with $i=\sqrt {-1}$ which we still denote by $x$. With this notation we have $x^{\bot}=ix$. Assume that the points $a_i$ are in general position (so that no three of them are aligned). We consider an one-variable polynomial of degree $d-1$ defined by  
\begin{equation}\label{eqn:h.st}
h_{st}(w)=\prod_{m=0,m\ne s}^{d-1}\big (w-\langle (a_s-a_t)^{\bot},a_m\rangle \big),\,\,w\in\RR,\,\, s\ne t,\,\,0\leq s,t\leq d-1.
\end{equation}
The polynomial $h_{st}$ appears in the formulas for Kergin and Hakopian interpolation polynomials and 
plays an important role in our arguments. It is worth pointing out that $h_{st}$ is a multiple of the polynomial $q_{ts}$ used in relation (2.18) in \cite{boscalvi1} where basic properties of $q_{ts}$ are established. To make our exposition self-contained we state and prove a few properties of $h_{st}$. 
\begin{lemma}\label{lem:h.st}
Let $d\geq 3$ and let $A=(a_0,a_1,\ldots,a_{d-1})$ be a d-tuple of points in general position in $\RR^2$. Then
\begin{enumerate}
\item $h_{st}\big(\langle(a_s-a_t)^{\bot},a_v \rangle \big)=0$,\quad $0\leq v\leq d-1$;
\item $h'_{st}\big(\langle(a_s-a_t)^{\bot},a_s \rangle\big)=\prod_{m=0,m\ne s,t}^{d-1} \langle (a_s-a_t)^{\bot},(a_s-a_m)\rangle$;
\item If $u, v$ and $s$ are pairwise distinct and $\langle (a_s-a_t)^{\bot},(a_u-a_v)\rangle=0$, then\\ $h'_{st}\big(\langle(a_s-a_t)^{\bot},a_u \rangle\big)=0$. 
\end{enumerate}
\end{lemma}
\begin{proof}
Observe that $h_{st}\big(\langle (a_s-a_t)^{\bot},a_v\rangle\big)=\prod_{m=0,m\ne s}^{d-1}\langle (a_s-a_t)^{\bot},(a_v-a_m)\rangle$. The product in the  right hand side has the vanishing factor $\langle (a_s-a_t)^{\bot},a_v-a_v\rangle$ when $v\ne s$ and the vanishing factor $\langle (a_s-a_t)^{\bot},a_s-a_t\rangle$ when $v=s$. Thus $h_{st}(\langle (a_s-a_t)^{\bot},a_v\rangle)=0$. For the proof of assertion (2), we just compute the derivative of $h_{st}$, that is 
\begin{equation}\label{eqn:deriv.h.st.1}
h'_{st}(w)=\sum_{m=0,m\ne s}^{d-1}\prod_{j=0,j\ne m,s}^{d-1}\big (w- \langle (a_s-a_t)^{\bot},a_j\rangle\big ).
\end{equation}
It is easy to see that the vanishing factor $\langle (a_s-a_t)^{\bot},(a_s-a_t)\rangle$ is contained in the product $\prod_{j=0,j\ne m,s}^{d-1} \langle (a_s-a_t)^{\bot},(a_s-a_j)\rangle$ whenever $m\ne t$. Hence, in view of (\ref{eqn:deriv.h.st.1}), we have 
\begin{equation}\label{eqn:deriv.h.st.2}
h'_{st}\big(\langle (a_s-a_t)^{\bot},a_s\rangle \big)=\prod_{j=0,j\ne t,s}^{d-1} \langle (a_s-a_t)^{\bot},(a_s-a_j)\rangle.
\end{equation}
For the last assertion it is enough to verify that $\prod_{j=0,j\ne m,s}^{d-1} \langle (a_s-a_t)^{\bot},(a_u-a_j)\rangle=0$ for every $m\ne s$. To prove this we only notice that the product in the left hand side contains the vanishing factor $\langle (a_s-a_t)^{\bot},(a_u-a_u)\rangle$ if $m= v$ and the vanishing factor $\langle (a_s-a_t)^{\bot},(a_u-a_v)\rangle$ if $m\ne v$. 
\end{proof}
Theorem \ref{thm:formular-hakop} below gives a formula for Hakopian interpolation polynomial in $\RR^2$. It is similar to that of Kergin interpolation polynomial found by Bos and Calvi. Here, we denote by $cv(A)$ the convex hull of the set $A$. 
\begin{theorem}[Bos and Calvi]\label{thm:boscalvi} 
Let $A=(a_0,a_1,\ldots,a_{d-1})$ be a tuple of $d$ points in general position in the plane. Then the (extended) Kergin operator $\mathcal K[A]$ is continuously defined on $C^{1}(cv(A))$ by the formula
$$\mathcal K[A;f]=\sum_{j=0}^{d-1}f(a_j)P_j+\sum_{0\leq s<t\leq d-1}P_{st}\int\limits_{[a_s,a_t]}\textup{D}_{(a_t-a_s)^{\bot}}f ,$$
where $P_j$ is the real part of the $(j+1)$-st fundamental Lagrange polynomial corresponding to the complex nodes $a_0,\ldots,a_{d-1}$, that is
\begin{equation}\label{eqn:pj.general}
P_j(x^1,x^2)=\Re\big(\prod_{m=0,m\ne j}^{d-1}\frac{(x^1+ix^2)-a_m}{a_j-a_m}\big),
\end{equation}
and
\begin{equation}\label{eqn:pst.general}
P_{st}(x)=\frac{h_{st}\big( \langle (a_s-a_t)^{\bot},x\rangle\big)}{\|a_s-a_t\|^2h'_{st}\big( \langle (a_s-a_t)^{\bot},a_s\rangle\big)}.
\end{equation}
\end{theorem}
\begin{theorem}\label{thm:formular-hakop}
Let $d\geq 2$ and let $A=(a_0,a_1,\ldots,a_{d-1})$ be a tuple of $d$ points in general position in the plane. Then the (extended) Hakopian operator $\mathcal H[A]$ is continuously defined on $C(cv(A))$ by the formula
\begin{equation}\label{for:hakopian}
\mathcal {H}[A;f]=\sum_{0\leq s<t\leq d-1} Q_{st} \int\limits_{[a_s,a_t]}f,
\end{equation}
where 
\begin{equation}\label{for:basic-hakop}
Q_{st}(x)=\frac{h'_{st}\big( \langle (a_s-a_t)^{\bot},x\rangle\big)}{h'_{st}\big( \langle (a_s-a_t)^{\bot},a_s\rangle\big)}.
\end{equation}
\end{theorem}
\begin{proof} Of course, when $d=2$ then $Q_{01}=1$ and (\ref{for:hakopian}) is trivial. Now, suppose that $d\geq 3$. Since, for all $0\leq s<t\leq d-1$, $Q_{st}$ is a polynomial of degree at most $d-2$, the polynomial defined in the right hand side of (\ref{for:hakopian}) belongs to $\mathcal P_{d-2}(\RR^2)$. Let us call $H$ this polynomial. Thanks to (\ref{for:reduce-hakop}), we have to show that 
\begin{equation}\label{eqn:hakop.condi}
\int\limits_{[a_u,a_v]}H=\int\limits_{[a_u,a_v]}f,\quad\text{for all }\,0\leq u<v\leq d-1 .
\end{equation}
In view of (\ref{for:hakopian}), it suffices to verify that    
\begin{equation}\label{for:remain-equal}
\int\limits_{[a_u,a_v]} Q_{st}=\delta_{us}\delta_{vt}, \quad\text{for all }\, 0\leq s<t\leq d-1,\,\,\, 0\leq u<v\leq d-1,
\end{equation}
where $\delta$ is the Kronecker symbol. Looking at (\ref{for:basic-hakop}), we have  
\begin{multline}\label{for:basic-integral}
\int\limits_{[a_u,a_v]} Q_{st}
=\int\limits_{0}^1 Q_{st}(a_u+w(a_v-a_u))\textup{d}w \\
=\frac{1}{h'_{st}\big( \langle (a_s-a_t)^{\bot},a_s\rangle\big)}\int\limits_0^1 h'_{st}\big(\langle (a_s-a_t)^{\bot},a_u\rangle+ w\langle (a_s-a_t)^{\bot},(a_v-a_u)\rangle  \big)\textup{d}w.
\end{multline}  
To deal with the last integral we examine three cases. \\
First, we assume that $(s,t)=(u,v)$. Then, since $\langle (a_s-a_t)^{\bot},(a_t-a_s)\rangle=0$, relation (\ref{for:basic-integral}) gives 
$$\int\limits_{[a_s,a_t]} Q_{st}= 1.$$
The second case occurs when $(s,t)\ne (u,v)$ and $\langle (a_s-a_t)^{\bot},(a_v-a_u)\rangle=0$. Then $s\ne u$ and $s\ne v$. Indeed, if, for exemple, $s=u$, then the relation $\langle (a_s-a_t)^{\bot},(a_v-a_s)\rangle=0$ implies that $a_s, a_t$ and $a_v$ are collinear, contrary to the hypothesis. Now, the integral term in  (\ref{for:basic-integral}) reduces to 
$$\int\limits_{[a_u,a_v]} Q_{st}=\frac{h'_{st}(\langle (a_s-a_t)^{\bot},a_u\rangle )}{h'_{st}\big( \langle (a_s-a_t)^{\bot},a_s\rangle\big)}=0,$$
where we use Lemma \ref{lem:h.st}(3) in the second equality. The last case is when $(s,t)\ne (u,v)$ and $\langle (a_s-a_t)^{\bot},(a_v-a_u)\rangle\ne 0$ then, calculating the intergral (\ref{for:basic-integral}), we have
\begin{equation}\label{eqn:int.basic.3rdcase}
\int\limits_{[a_u,a_v]} Q_{st}=\frac{h_{st}(\langle (a_s-a_t)^{\bot},a_v\rangle)-h_{st}(\langle (a_s-a_t)^{\bot},a_u\rangle)}{h'_{st}\big( \langle (a_s-a_t)^{\bot},a_s\rangle\big)\langle (a_s-a_t)^{\bot},(a_v-a_u)\rangle}.
\end{equation}  
Now, Lemma \ref{lem:h.st}(1) follows that the right hand side of (\ref{eqn:int.basic.3rdcase}) vanishes and the proof is complete. 
\end{proof}
\subsection{} We use the above formulas to establish multivariate analogues of the classical Lebesgue inequality for Lagrange interpolation. For the proof, we refer to  \cite[Theorem 1.1]{boscalvi1} and \cite[Theorem 5]{liang.lu}. 
\begin{lemma}\label{lem:lebesgue.kergin}
Let $A=(a_0,a_1,\ldots,a_{d-1})$ be a tuple of $d$ points in general position in the plane, and let $K\subset\RR^2$ be a convex compact set containing $A$. Then for every $f\in  C^1(K)$ and every $Q_{d-1}\in \mathcal P_{d-1}(\RR^2)$ we have
\begin{equation}\label{eqn:lebesgue.kergin}
\|f-\mathcal K[A;f]\|_{K}\leq \big(1+\sum_{j=0}^{d-1}\|P_j\|_{K}\big)\|f-Q_{d-1}\|_{K}+\textup{diam}(K)\sum_{0\leq s<t\leq d-1} \|P_{st}\|_{K} \|\textup{D}f-\textup{D}Q_{d-1}\|_{K},
\end{equation}
where $\textup{diam}(K)$ is the diameter of $K$, the polynomials $P_j$ and $P_{st}$ are defined in (\ref{eqn:pj.general}) and (\ref{eqn:pst.general}) respectively. 
\end{lemma}
\begin{lemma}\label{lem:lebesgue.hakopian}
Let $A=(a_0,a_1,\ldots,a_{d-1})$ be a tuple of $d$ points in general position in the plane, and let $K\subset\RR^2$ be a convex compact set containing $A$. Then for every $f\in  C(K)$ and every $Q_{d-2}\in \mathcal P_{d-2}(\RR^2)$ we have 
\begin{equation}\label{eqn:lebesgue.hakopian}
\|f-\mathcal H[A;f]\|_{K}\leq \big(1+\sum_{0\leq s<t\leq d-1} \|Q_{st}\|_{K}\big) \|f-Q_{d-2}\|_{K},
\end{equation}
where the polynomials $Q_{st}$ are defined in (\ref{for:basic-hakop}).
\end{lemma}
\section{Kergin and Hakopian interpolants at Leja sequences for the disk}
  
\begin{definition}
Let $D$ be the closed unit disk in the complex plane and $E=(e_n: n\in \NN)$ be a sequence of points in $D$. One says that $E$ is a Leja sequence for $D$ if the following property hold true,
$$|\prod_{j=0}^{d-1}(e_d-e_j)|=\max_{z\in D}|\prod_{j=0}^{d-1}(z-e_j)|,\quad \text{for all }d\geq 1.$$
\end{definition}
A $d$-tuple $E_d=(e_0, e_1,\ldots,e_{d-1})$ is called a $d$-Leja section.  In this paper we only consider Leja sequences whose first entry is equal to 1. It is not diffcult to describe the structure of Leja sequences for $D$. The following theorem is proved in \cite{biacal}.
\begin{theorem}[Bia{\l}as-Cie{\.z} and Calvi]\label{structure-Leja}
The structure of a Leja sequence $E=(e_n:n\in\NN)$ for the unit disk $D$ with $e_0=1$ is given by the following rules. 
\begin{enumerate}
\item The underlying set of the $2^n$-Leja section $E_{2^n}$ consists of the $2^n$-th roots of unity 
\item The $2^{n+1}$-Leja section is $(E_{2^n},\rho E^{(1)}_{2^n})$,  where $\rho$ is a $2^n$-th roots of -1 and $E^{(1)}_{2^n}$ is the $2^n$-Leja section of a Leja sequence $E^{(1)}=(e^{(1)}_n:n\in\NN)$ for the unit disk with $e^{(1)}_0=1$.  
\end{enumerate} 
\end{theorem} 
  Next, we use Lebesgue-type inequalities for Kergin and Hakopian interpolants along with the method of Bos and Calvi to prove the following convergence results.
\begin{theorem}\label{maintheorem1}
 Let $\mathcal K[E_d;f]$ denote the Kergin interpolation polynomial of $f$ with respect to the Leja section $E_d=(e_0,\ldots,e_{d-1})$ of a Leja sequence $E=(e_n:n\in\NN)$ for $D$.
\begin{enumerate}
\item If $f\in C^4(D)$, then $\mathcal K[E_d;f]$ converges uniformly to $f$ on $D$ as $d\to\infty$;
\item If $f\in C^\infty(D)$, then $\textup{D}^\beta\big(\mathcal K[E_d;f]\big)$ converges uniformly to $\textup{D}^\beta f$ on $D$ as $d\to\infty$, for every two-dimensional index $\beta$. 
\end{enumerate}
 \end{theorem}
 \begin{corollary}\label{cor.expand.ker} For every $f\in C^\infty(D)$, the series
 $$\sum_{d=0}^\infty \int\limits_{[e_0,\dots,e_d]}  \textup{D}^df(\cdot, x-e_0,\dots,x-e_{d-1})$$
converges to $f$ uniformly on $D$. Moreover, the convergence extends to all derivatives. 
 \end{corollary}
 \begin{proof} In view of Newton's formula for Kergin interpolation (see \cite[Theorem 2]{micchelli1}), the $(d+1)$-st partial sum of the series is exactly $\mathcal K[E_d;f]$. \end{proof}
\begin{theorem}\label{maintheorem2}
 Let $\mathcal H[E_d;f]$ denote the Hakopian interpolation polynomial of a function $f$ with respect to the Leja section $E_d=(e_0,\ldots,e_{d-1})$ of a Leja sequence $E=(e_n:n\in\NN)$ for $D$.
\begin{enumerate}
\item If $f\in C^5(D)$, then $\mathcal H[E_d;f]$ converges uniformly to $f$ on $D$ as $d\to\infty$;
\item If $f\in C^\infty(D)$, then $\textup{D}^\beta\big(\mathcal H[E_d;f]\big)$ converges uniformly to $\textup{D}^\beta f$ on $D$ as $d\to\infty$, for every two-dimensional index $\beta$. 
\end{enumerate}       
\end{theorem}
 \begin{corollary}\label{cor.expand.hako} 
For every $f\in C^\infty(D)$, the series
 $$\sum_{d=1}^\infty \sum_{0\leq j_1<j_2<\cdots<j_{d-1}\leq d-1}\int\limits_{[e_0,\dots,e_d]}  \textup{D}^df(\cdot, x-e_{j_1},\dots,x-e_{j_{d-1}}),$$ 
converges to $f$ uniformly on $D$. Moreover, the convergence extends to all derivatives. 
 \end{corollary}
\begin{proof} Looking at the formula for Hakopian interpolation (see \cite{hakopian, goodman}), the $d$-th partial sum of the series is exactly $\mathcal H[E_d;f]$. \end{proof}
\begin{remark} We denote by $F_p$ the set of functions from $\{0, \dots, d-1\}$ to $\{1,2\}$.  For $\tau\in F_d$, we set $\alpha(\tau)=(a,b)$ with $a$ (resp. $b$) the number of times that $\tau$ takes on the value $1$ (resp. the value $2$) and we write $(x-e)^\tau:= \prod_{i=0}^{d-1} (x-e_i)_{\tau(i)}$,
where $(x-e_i)_{\tau(i)}$ is $x^1-\Re e_i$ (resp. $x^2-\Im e_i$) if $\tau(i)=1$ (resp. $\tau(i)=2$). In particular, $(x-e)^\tau$ is a polynomial of degree $d$. We have
\begin{equation}\label{eqn:total.derive.expan}
\textup{D}^df(\cdot, x-e_0,\dots,x-e_{d-1})=\sum_{\tau \in F_p } \textup{D}^{\alpha(\tau)}f(\cdot) (x-e)^\tau.\end{equation}
The series expansion in Corollary \ref{cor.expand.ker} can be rewriten as 
 \begin{equation}
f(x)=\sum_{d=0}^\infty\sum_{\tau\in F_d} \int\limits_{[e_0,\dots,e_d]}  \textup{D}^{\alpha(\tau)}f(\cdot) (x-e)^\tau.  
\end{equation} 
Thus the polynomials $(x-e)^\tau$ can be regarded as a generalization of the classical Newton polynomials and the above expansion as a multivariate Newton series expansion.  A similar observation could be done with Corollary \ref{cor.expand.hako}.
\end{remark}

 In the rest of this note we always interpolate at Leja sections $E_d$. In order to use the Lebesgue-type inequalities given in (\ref{eqn:lebesgue.kergin}) and (\ref{eqn:lebesgue.hakopian}), we need a kind of Jackson theorem that we now recall.

For $f\in C^k(D)$, we set
$$\|f\|_k:=\sum_{|\beta|\leq k}\|\textup{D}^\beta f \|_{D},$$
$$ \omega(f;\delta):=\sup\{|f(x)-f(y)|: \|x-y\|\leq \delta,\,\, x,y\in D\},\,\,\,\omega(f^{(k)};\delta)=\sum_{|\beta|=k}\omega(\textup{D}^\beta f;\delta),\,\,\delta>0.$$ 
The following theorem, proved in \cite[p. 164]{ragozin}, is due to Ragozin.
\begin{theorem}[Ragozin]\label{ragozin}
Given $f$ in $C^k(D)$ with $k\geq 0$, there exists polynomials $Q_d$ with $\deg Q_d\leq d$ such that
$$\|f-Q_d\|_{D}\leq M(k) d^{-k}\big [d^{-1}\|f\|_k+\omega(f^{(k)};1/d)\big ],$$
where $M(k)$ is a positive constant which depends only on $k$. 
\end{theorem}
In \cite{boscalvi1}, Bos and Calvi  slightly modified a method of Ragozin and used Theorem \ref{ragozin} to prove a simultaneously approximate theorem for $C^2$ functions on $D$ (see \cite[Lemma 4.1]{boscalvi1}). Examining their proof, we see that the proof  easily extends to $C^k$ functions. We state the generalized result without proof.   
\begin{lemma}\label{jackson}
Given $f$ in $C^k(D)$ with $k\geq 1$, there exists a sequence of polynomials $Q_d$ with $\deg Q_d\leq d$ such that
$$\lim_{d\to\infty} d^{k-1}\|f-Q_d\|_{D}=0\quad\text{and}\quad\lim_{d\to\infty} d^{k-1}\|\textup{D}f-\textup{D}Q_d\|_{D}=0.$$
\end{lemma}

\noindent
Next, we investigate the growth of Lebesgue-type constants  
$$\sum_{j=0}^{d-1}\|P_j\|_{D},\quad \sum_{0\leq s<t\leq d-1}\|P_{st}\|_{D}\quad\text{and}\quad \sum_{0\leq s<t\leq d-1}\|Q_{st}\|_{D}.$$
Looking at the formula for $P_j$ in Theorem \ref{thm:boscalvi}, we see that $\sum_{j=0}^{d-1}\|P_j\|_{D}$ is dominated by the Lebesgue constant $\Delta(E_d)$ for Lagrange interpolation corresponding $d$ complex nodes $e_j$, $0\leq j\leq d-1$, over $D$. But \cite[Corollary 7]{jpcpvm} tells us that $\Delta(E_d)=O(d\log d)$ as $d\to\infty$. It gives us the estimate
\begin{equation}\label{eqn:first.lebesgue.type}
\sum_{j=0}^{d-1}\|P_j\|_{D}=O(d\log d)\text{ as }d\to \infty.
\end{equation}     
The estimates for the remaining Lebesgue-type constants are simple consequences of the following two theorems that are proved in the last section. Here, in the formulas for $h_{st}$ and $h'_{st}$, we take $a_j=e_j$ so that
\begin{equation}\label{eqn:hst.new}
 h_{st}(w)=\prod_{m=0,m\ne s}^{d-1}\big (w-\langle (e_s-e_t)^{\bot},e_m\rangle \big), \quad w\in\RR, 
\end{equation}
\begin{equation}\label{eqn:deri.hst.new}
h'_{st}\big(\langle (e_s-e_t)^{\bot},e_s\rangle \big)=\prod_{m=0,m\ne s,t}^{d-1} \langle (e_s-e_t)^{\bot},(e_s-e_m)\rangle.
\end{equation}
\begin{theorem}\label{thm.second.lebesgue.type}
We have $\|P_{st}\|_{D}\leq 2d$ for all $d\geq 2$ and $0\leq s<t\leq d-1$, where
$$P_{st}(x)=\frac{h_{st}\big( \langle (e_s-e_t)^{\bot},x\rangle\big)}{|e_s-e_t|^2h'_{st}\big( \langle (e_s-e_t)^{\bot},e_s\rangle\big)},\quad\quad x\in D.$$
\end{theorem} 
\begin{theorem}\label{thm.third.lebesgue.type}
We have $\|Q_{st}\|_{D}\leq 4d^3$ for all $d\geq 2$ and $0\leq s<t\leq d-1$, where
$$Q_{st}(x)=\frac{h'_{st}\big( \langle (e_s-e_t)^{\bot},x\rangle\big)}{h'_{st}\big( \langle (e_s-e_t)^{\bot},e_s\rangle\big)}, \quad\quad x\in D.$$
\end{theorem} 
Next, we note that $D$ satisfies a Markov inequality, that is, 
\begin{equation}\label{eqn:markov.D}
\max\left\{\left\|\frac{\partial p}{\partial x^1}\right\|_D, \left\|\frac{\partial p}{\partial x^2}\right\|_D\right\}\leq (\deg p)^2 \|p\|_D,\quad p\in\mathcal P(\RR^2),
\end{equation}
see \cite{sarantopoulos}. Repeated application of (\ref{eqn:markov.D}) yields estimates for each partial derivatives,
\begin{equation}\label{ineq.markov}
\|\textup{D}^\beta p\|_D\leq (\deg p)^{2|\beta|}\|p\|_D, \quad p\in\mathcal P(\RR^2).
\end{equation}
\begin{proof}[Proof of Theorem \ref{maintheorem1}]
Using Lemma \ref{jackson} for $f\in C^k(D)$ we can find a sequence of polynomials $Q_{d-1}\in\mathcal{P}_{d-1}(\RR^2)$, $d\in \mathbb{N}^\star$, and a sequence of positive numbers $(\epsilon_n:n\in\NN)$ that converges to 0 such that  
\begin{equation}
\max\{\|f-Q_{d-1}\|_{D}, \|\textup{D}f-\textup{D}Q_{d-1}\|_{D}\}\leq \frac{\epsilon_{d}}{d^{k-1}}, \quad d\geq 1. 
\end{equation}
 But Theorem \ref{thm.second.lebesgue.type} gives $\sum_{0\leq s<t\leq d-1}\|P_{st}\|_{D}\leq d^2(d-1)$ and (\ref{eqn:first.lebesgue.type}) gives $1+ \sum_{j=0}^{d-1}\|P_j\|_{D}=O(d\log d)$. It follows from (\ref{eqn:lebesgue.kergin}) that
\begin{equation}\label{error.ker}
\|f-\mathcal K[E_d;f]\|_D\leq \Big(O(d\log d)+2d^2(d-1) \Big)\frac{\epsilon_d}{d^{k-1}}\leq\frac{M\epsilon_d}{d^{k-4}}, 
 \end{equation}
where $M$ is a constant. The right hand side of (\ref{error.ker}) tends to 0 as $d\to\infty$ when $f\in C^4(D)$, i.e., $k=4$. This follows the first assertion. To prove the second one with the hypothesis that $f\in C^\infty(D)$, we first observe that 
\begin{equation}
\|\mathcal K[E_{n+1};f]-\mathcal K[E_{n};f] \|_D\leq \|f-\mathcal K[E_{n+1};f]\|_D+\|f-\mathcal K[E_{n};f] \|_D\leq \frac{M(\epsilon_{n+1}+\epsilon_{n})}{n^{k-4}}. 
\end{equation}
Applying Markov's inequality in (\ref{ineq.markov}) for $p=\mathcal K[E_{n+1};f]-\mathcal K[E_{n};f]\in \mathcal P_{n+1}(\RR^2)$ we obtain
\begin{equation}
\Big\|\textup{D}^{\beta}\Big(\mathcal K[E_{n+1};f]-\mathcal K[E_{n};f] \Big)\Big\|_D\leq \frac{M(\epsilon_{n+1}+\epsilon_n)(n+1)^{2|\beta|}}{n^{k-4}}, \quad \beta\in\NN^2. 
\end{equation}  
Now, we choose $k=k(\beta)=2|\beta|+6$. Then the series 
$$\sum_{n=1}^{\infty}\frac{M(\epsilon_{n+1}+\epsilon_n)(n+1)^{2|\beta|}}{n^{k-4}}$$ 
converges. This follows the uniform convergence on $D$ of the series 
$$\textup{D}^{\beta}\Big(\mathcal K[E_1;f]\Big)+\sum_{n=1}^\infty \textup{D}^{\beta}\Big(\mathcal K[E_{n+1};f]-\mathcal K[E_{n};f] \Big).$$
Hence $\textup{D}^{\beta}\Big(\mathcal K[E_d;f]\Big)$ converges uniformly on $D$ as $d\to\infty$, for every $\beta\in\NN^2$. A classical reasoning show that if  $\mathcal K[E_d;f]$ converges uniformly to $f$ on $D$ and $\textup{D}^{\beta}\Big(\mathcal K[E_d;f]\Big)$ converges uniformly on $D$ for every $\beta\in\NN^2$, then
$$\textup{D}^{\beta}\Big(\mathcal K[E_d;f]\Big)\to \textup{D}^\beta f,\quad\text{uniformly on }D,\quad \text{for every }\beta\in\NN^2.$$  
\end{proof}      
\begin{proof}[Proof of Theorem \ref{maintheorem2}]
Using Theorem \ref{ragozin} for $f\in C^k(D)$ we can find a sequence of polynomials $Q_{d-2}\in\mathcal P_{d-2}(\RR^2)$, $d\geq 2$, and a sequences of positive numbers $(\delta_n:n\in\NN)$ that converges to 0 such that  
$$ \|f-Q_{d-2}\|_{D}\leq \frac{\delta_d}{d^{k}},\quad d\geq 2.$$ 
From Theorem \ref{thm.third.lebesgue.type} we have $ \sum_{0\leq s<t\leq d-1}\|Q_{st}\|_{D}\leq 2d^4(d-1)<2d^5-1$. Hence, in view of Lemma \ref{lem:lebesgue.hakopian}, we get
$$\|f-\mathcal H[E_d;f]\|_D\leq \big (1+\sum_{0\leq s<t\leq d-1}\|Q_{st}\|_{D}\big)\|f-Q_{d-2}\|_{D}\leq \frac{2\delta_d}{d^{k-5}}.$$
This estimate is the same as (\ref{error.ker}).  
The conclusions of the theorem now follow by repeating the arguments in the proof of Theorem \ref{maintheorem1}.
\end{proof} 
\section{Further properties of Leja sequences for the disk}
\subsection{Decomposition of Leja sections}
Let $E=(e_n:n\in\NN)$ be a Leja sequence for $D$. As observed in \cite{calvimanh2}, repeated applications of the rule in Theorem \ref{structure-Leja} show that if $d=2^{n_0}+2^{n_1}+\dots+2^{n_r}$ with $n_0>n_1>\dots > n_r\geq 0$, then 
\begin{align} E_{d}&=(E_{2^{n_0}}\, ,\, \rho_0 E^{(1)}_{d-2^{n_0}})=(E_{2^{n_0}}\, ,\, \rho_0 E^{(1)}_{2^{n_1}}\, , \, \rho_1\rho_0 E^{(2)}_{d-2^{n_0}-2^{n_1}})
\\&=\dots =(E_{2^{n_0}}\, ,\, \rho_0 E^{(1)}_{2^{n_1}}\, , \, \rho_1\rho_0 E^{(2)}_{2^{n_2}}, \dots ,\, \rho_{r-1}\dots\rho_1\rho_0 E^{(r)}_{2^{n_r}}), \label{eqn:define.rho}\end{align}
where each $E^{(j)}_{2^{n_j}}$ consists of a complete set of the $2^{n_j}$-roots of unity, arranged in a certain order, and $\rho_j$ satisfies $\rho_j^{2^{n_j}}=-1$ for all $0\leq j\leq r-1$.

 For a sequence of complex numbers $Z=(z_k: k\in\NN )$ we define $Z(j:k):=(z_j,z_{j+1},\ldots,z_k)$, of course, $Z(0:k-1)=Z_{k}$.  
For $d=2^{n_0}+2^{n_1}+\dots+2^{n_r}$ with $n_0>n_1>\dots > n_r\geq 0$, let us set 
\begin{equation}\label{eqn:set.dj}
d_{-1}=0,\quad d_j=2^{n_0}+\cdots+2^{n_j},\quad \text{for }0\leq j\leq r.
\end{equation} 
With this notation, in view of (\ref{eqn:define.rho}), we have 
\begin{equation}\label{eqn.decom.leja}
E(d_{-1}:d_0-1)=E_{2^{n_0}}\quad\text{and}\quad E(d_j:d_{j+1}-1)=\rho_j\cdots\rho_0 E^{(j+1)}_{2^{n_{j+1}}},\,\,\, 0\leq j\leq r-1.
\end{equation}
 From now on, we always denote $\theta_n:=\arg e_n$ for $n\geq 0$ and $\varphi_j:=\arg \rho_j$ for $0\leq j\leq r-1$. Since $\rho_j^{2^{n_j}}=-1$ we may put $\varphi_j=\frac{(2q_j+1)\pi}{2^{n_j}}$, $q_j\in\NN$. The following lemma is similar to \cite[Lemma 3]{jpcpvm}.  

\begin{lemma}\label{lem:simplify.factor}
Let $E=(e_n:n\in\NN)$ be a Leja sequence for $D$ and $d=2^{n_0}+2^{n_1}+\dots+2^{n_r}$ with $n_0>n_1>\dots > n_r\geq 0$. Then 
\begin{enumerate}
\item $\prod_{m=d_{-1}}^{d_0-1}|z-e_m|=|z^{2^{n_0}}-1|$;
\item $ \prod_{m=d_j}^{d_{j+1}-1} |z-e_m|=|(z\rho_0^{-1}\cdots\rho_j^{-1})^{2^{n_{j+1}}}-1|$,\quad $0\leq j\leq r-1$;
\item $\prod_{m=d_j,m\ne k}^{d_{j+1}-1} |e_k-e_m|=2^{n_{j+1}}$,\quad $d_j\leq k\leq d_{j+1}-1$, \quad $-1\leq j\leq r-1$.
\end{enumerate}
\end{lemma}
\begin{proof}
Since $E^{(j+1)}_{2^{n_{j+1}}}$ forms a complete set of the $2^{n_{j+1}}$-st roots of unity, the relation in (\ref{eqn.decom.leja}) gives 
$$\prod_{m=d_j}^{d_{j+1}-1} |z-e_m|=\prod_{e\in E^{(j+1)}_{2^{n_{j+1}}}}|z-\rho_j\cdots\rho_0e|=\prod_{e\in E^{(j+1)}_{2^{n_{j+1}}}}|(\rho_j\cdots\rho_0)^{-1}z-e|=|(z\rho_0^{-1}\cdots\rho_j^{-1})^{2^{n_{j+1}}}-1|.$$ 
This proves the first two assertions. For the third one, we observe that $e_k=\rho_j\cdots\rho_0e'$ with $e'\in E^{(j+1)}_{2^{n_{j+1}}}$. It follows that
\begin{equation*} 
\prod_{m=d_j,m\ne k}^{d_{j+1}-1} |e_k-e_m|=\prod_{e\in E^{(j+1)}_{2^{n_{j+1}}}, e\ne e'}|e'-e|=2^{n_{j+1}},
\end{equation*}
since the middle term is the modulus of the derivative of $z^{2^{n_{j+1}}}-1$ at $e'$. This completes the proof.
\end{proof}
\subsection{Some trigonometric inequalities}
Let $T_n$ be the {\em monic} Chebyshev polynomial of degree $n$, that is $2^{n-1}T_n(\cos\varphi)=\cos(n\varphi)$. If $\cos(n\beta)\ne \pm1$, then the equation $T_n(x)=T_n(\cos\beta)$ has $n$ distinct roots: $\cos(\beta+2m\pi/n)$, $m=0,\ldots,n-1$. Hence
\begin{equation}\label{eqn:chebyshev.repren}
T_n(\cos \varphi)-T_n(\cos\beta)=\prod_{m=0}^{n-1}[\cos\varphi-\cos(\beta+2m\pi/n)].
\end{equation}
Since both sides of (\ref{eqn:chebyshev.repren}) are continuous functions of $\beta$, relation (\ref{eqn:chebyshev.repren}) holds true for all $\beta\in\RR$. Now, the relation in (\ref{eqn.decom.leja}) implies that, for $-1\leq j\leq r-1$, 
\begin{equation}\label{eqn:argument.fract}
\{\theta_m: d_j\leq m\leq d_{j+1}-1\}=\{\varphi_0+\cdots+\varphi_j+2\pi k/2^{n_{j+1}}[2\pi]: 0\leq k\leq 2^{n_{j+1}}-1\}, 
\end{equation}
where we write $\alpha=\beta [2\pi]$ if $\alpha=\beta \textup{ mod } 2\pi$ and the $\varphi_j$'s do not appear when $j=-1$. Using (\ref{eqn:chebyshev.repren}) for $n=2^{n_{j+1}}$ and $\beta=\varphi_0+\cdots+\varphi_j-\psi$ we obtain from (\ref{eqn:argument.fract}) the following result. 
\begin{lemma}\label{lem:simplify.factor.triog}
Let $E=(e_n:n\in\NN)$ be a Leja sequence for $D$, $\theta_n=\textup{arg}\,e_n$, and $d=2^{n_0}+2^{n_1}+\dots+2^{n_r}$ with $n_0>n_1>\dots > n_r\geq 0$. Then for all $\psi\in\RR$ and $-1\leq j\leq r-1$ we have
$$\prod_{m=d_{j}}^{d_{j+1}-1}[\cos\varphi-\cos(\theta_m-\psi)]=\frac{1}{2^{2^{n_j+1}-1}}\Big (\cos 2^{n_{j+1}}\varphi-\cos 2^{n_{j+1}}(\varphi_0+\cdots+\varphi_j-\psi)\Big),$$
where the $d_i$'s are defined in (\ref{eqn:set.dj}).
\end{lemma}
\begin{lemma}\label{lem:prod.one.missing}
If $n\geq 1$, $0\leq j\leq n-1$ and $\sin(\beta+2j\pi/n)\ne 0$, then
$$\prod_{m=0,m\ne j}^{n-1} |\cos\varphi-\cos(\beta+2m\pi/n)|\leq \frac{2n}{2^{n-1}|\sin(\beta+2j\pi/n)|},\quad \varphi\in\RR.$$ 
\end{lemma}
\begin{proof}
In view of (\ref{eqn:chebyshev.repren}) we have
\begin{equation}\label{eqn:prod.one.missing}
\prod_{m=0,m\ne j}^{n-1} |\cos\varphi-\cos(\beta+2m\pi/n)|=\Big |\frac{\cos(n\varphi)-\cos(n\beta) }{2^{n-1}[\cos\varphi-\cos(\beta+2j\pi/n)]}\Big |.
\end{equation}
Set $\psi=\beta+2j\pi/n$, $\psi_1=(1/2)(\varphi+\psi), \psi_2=(1/2)(\varphi-\psi)$. Then the sum-to-product formula for cosines transforms the right hand side of (\ref{eqn:prod.one.missing}) into $2^{-n+1}|\sin n\psi_1\sin n\psi_2|/|\sin\psi_1\sin\psi_2|$. Since $\sin\psi=\sin\psi_1\cos\psi_2-\cos\psi_1\sin\psi_2$, we obtain after simplication 
\begin{eqnarray*}
\prod_{m=0,m\ne j}^{n-1} |\cos\varphi-\cos(\beta+2m\pi/n)|
&=&\frac{|\sin\psi_1\cos\psi_2-\cos\psi_1\sin\psi_2|\cdot|\sin n\psi_1\sin n\psi_2|}{2^{n-1}|\sin\psi||\sin\psi_1\sin\psi_2|}\\
&\leq&\frac{1}{2^{n-1}|\sin\psi|}\Big(\frac{|\sin n\psi_2|}{|\sin\psi_2|}+\frac{|\sin n\psi_1|}{|\sin\psi_1|}\Big)\\
&\leq& \frac{2n}{2^{n-1}|\sin(\beta+2j\pi/n)|},
\end{eqnarray*}
where we use the classical inequality $|\sin n\alpha|\leq n|\sin\alpha|$ for $\alpha\in\RR$, $n\in\NN$ in the third line.
\end{proof}
Now, for $d_{j}\leq s\leq d_{j+1}-1$, Lemma \ref{lem:prod.one.missing} and equation (\ref{eqn:argument.fract}) imply that
\begin{equation}\label{eqn:one.missing.gener}
\prod_{m=d_{j}, m\ne s}^{d_{j+1}-1}[\cos\varphi-\cos(\theta_m+\beta)]\leq \frac{2\cdot 2^{n_{j+1}}}{2^{2^{n_j+1}-1}|\sin(\theta_s+\beta)|},\quad\varphi\in\RR.
\end{equation}
In (\ref{eqn:one.missing.gener}), taking $\beta=-\frac{\theta_s+\theta_t}{2}$ with $s<t$, we get the following result.
\begin{lemma}\label{lem:one.argument.missing}
Under the same assumptions of Lemma \ref{lem:simplify.factor.triog}, if $d_{j}\leq s\leq d_{j+1}-1$ with $-1\leq j\leq r-1$ and $s<t\leq d-1$, then
$$\prod_{m=d_{j}, m\ne s}^{d_{j+1}-1}[\cos\varphi-\cos(\theta_m-\frac{\theta_s+\theta_t}{2})]\leq \frac{2\cdot 2^{n_{j+1}}}{2^{2^{n_j+1}-1}|\sin\frac{\theta_s-\theta_t}{2}|}, \quad\varphi\in\RR.$$
\end{lemma}
\subsection{Further results} 
The following lemma will be used to get lower bounds for the denominators of $P_{st}$ and $Q_{st}$. 
\begin{lemma}\label{pro:produc-leja}
Let  $E=(e_n: n\in \NN)$ be a Leja sequence for $D$ and $d\geq 2$, $d=2^{n_0}+2^{n_1}+\cdots+2^{n_r}$ with $n_0>n_1>\cdots>n_r\geq 0$. Then $\prod_{m=0,m\ne s}^{d-1}|e_s-e_m|\geq 2^r$ for all $0\leq s\leq d-1$.
\end{lemma}
The proof of Lemma \ref{pro:produc-leja} requires a purely trigonometric inequality given in the following lemma. For the proof we refer the reader to \cite{calvimanh2}.
\begin{lemma}\label{lem:ineq+}
Let $r\geq 1$ and let $n_0>n_1>\cdots>n_r\geq 0$ be a decreasing sequence of natural numbers. If $\varphi_j=(2q_j+1)\pi/2^{n_j}$ with $q_j\in \ZZ$, $j=0,\ldots,r-1$, then 
\begin{equation}\label{for:trigono-inequ}
\prod_{j=0}^{r-1}|\sin 2^{n_{j+1}-1}(\varphi-\varphi_0-\cdots-\varphi_j)|\geq (1/2^{n_0-n_r})|\cos 2^{n_0-1}\varphi|,\quad\,\varphi\in\RR.
\end{equation}
\end{lemma}
\begin{proof}[Proof of Lemma \ref{pro:produc-leja}]
The case $r=0$ is trivial, since $\prod_{m=0,m\ne s}^{2^{n_0}-1}|e_s-e_m|=2^{n_0}$. Thus we may assume that $r\geq 1$. Notice that, since $\textup{arg}\,e_s=\theta_s$ and $\textup{arg}\,\rho_j=\varphi_j=(2q_j+1)\pi/2^{n_{j}}$, $0\leq j\leq r-1$, 
\begin{equation}\label{eqn:factor.to.trigon}
|(e_s\rho_0^{-1}\cdots\rho_{k}^{-1})^{2^{n_{k+1}}}-1|=2|\sin 2^{n_{k+1}-1}(\theta_s-\varphi_0-\cdots-\varphi_{k})|,\quad\quad 0\leq k\leq r-1.
\end{equation} 
First, suppose that $r\geq 2$ and $s\geq 2^{n_0}$. Then there exists a unique $0\leq j\leq r-1$ such that $d_j\leq s\leq d_{j+1}-1$, where the $d_j$'s are defined in (\ref{eqn:set.dj}). We write
\begin{equation}\label{eqn:three.factor}
\prod_{m=0,m\ne s}^{d-1}|e_s-e_m|=\prod_{m=d_{-1}}^{d_0-1}|e_s-e_m|\cdot\prod_{m=d_{j},m\ne s}^{d_{j+1}-1}|e_s-e_m|\cdot \prod_{k=0,k\ne j}^{r-1}\prod_{m=d_{k}}^{d_{k+1}-1}|e_s-e_m|.
\end{equation}
We will treat three factors in (\ref{eqn:three.factor}) independently. The first and the third part of Lemma \ref{lem:simplify.factor} give $\prod_{m=d_{-1}}^{d_0-1}|e_s-e_m|=|e_s^{2^{n_0}}-1|=2$, and $\prod_{m=d_{j},m\ne s}^{d_{j+1}-1}|e_s-e_m|=2^{n_{j+1}}$. On the other hand, the second part of Lemma \ref{lem:simplify.factor} along with equation (\ref{eqn:factor.to.trigon}) yields 
\begin{eqnarray}
\prod_{k=0,k\ne j}^{r-1}\prod_{m=d_{k}}^{d_{k+1}-1}|e_s-e_m|
&=&\prod_{k=0,k\ne j}^{r-1}|(e_s\rho_0^{-1}\cdots\rho_{k}^{-1})^{2^{n_{k+1}}}-1|\\
&=&2^{r-1} \prod_{k=0,k\ne j}^{r-1}|\sin 2^{n_{k+1}-1}(\theta_s-\varphi_0-\cdots-\varphi_{k})|.\label{eqn:trigo.prod.1}
\end{eqnarray}
Since $d_j\leq s\leq d_{j+1}-1$, relation (\ref{eqn:argument.fract}) tells us that $\theta_s=\varphi_0+\cdots+\varphi_{j}+2q\pi/2^{n_{j+1}}[2\pi]$ for some $q\in\ZZ$. To estimate the product in (\ref{eqn:trigo.prod.1}) we proceed  as in  \cite[Subsection 4.4]{calvimanh2}.  For the convenience of reader, we reproduce the proof .  For $0\leq k<j$, we have $\theta_s-\varphi_0-\cdots-\varphi_k=\varphi_{k+1}+\cdots+\varphi_{j}+2q\pi/2^{n_{j+1}}[2\pi]$. Hence the hypotheses on the values of $\varphi_0,\ldots,\varphi_{r-1}$ give   
\begin{equation}
2^{n_{k+1}-1}(\theta_s-\varphi_0-\cdots-\varphi_k)=\big((2q_{k+1}+1)\pi/2\big)[\pi].
\end{equation}
It follows that 
\begin{equation}\label{eqn:first-term}
\prod_{k=0}^{j-1}|\sin 2^{n_{k+1}-1}(\theta_s-\varphi_0-\cdots-\varphi_k)|=1.
\end{equation}
 On the other hand, using Lemma \ref{lem:ineq+} for $\varphi=\theta_s-\varphi_0-\cdots-\varphi_{j}$, i.e., $\varphi=2q\pi/2^{n_{j+1}}[2\pi]$, we obtain
\begin{equation}\label{eqn:second-term}
\prod_{k=j+1}^{r-1}|\sin 2^{n_{k+1}-1}(\theta_s-\varphi_0-\cdots-\varphi_k)| 
\geq (1/2^{n_{j+1}-n_r})|\cos 2^{n_{j+1}-1}\varphi|=1/2^{n_{j+1}-n_r}.
\end{equation}
Combining (\ref{eqn:first-term}) and (\ref{eqn:second-term}) we get
\begin{equation}\label{eqn:third-term}
\prod_{k=0,k\ne j}^{r-1}|\sin 2^{n_{k+1}-1}(\theta_s-\varphi_0-\cdots-\varphi_k)| \geq 1/2^{n_{j+1}-n_r}\geq 1/2^{n_{j+1}}.
\end{equation} 
Note that when $j=r-1$, then the left hand side of (\ref{eqn:third-term}) is equal to 1 and inequality (\ref{eqn:third-term}) is obviously true. When $j=0$, then the factor in (\ref{eqn:first-term}) does not appear. Now, using relation (\ref{eqn:third-term}) in (\ref{eqn:trigo.prod.1}) we get the following estimate 
$$ \prod_{k=0,k\ne j}^{r-1}\prod_{m=d_k}^{d_{k+1}-1}|e_s-e_m|\geq 2^{r-1} 2^{-n_{j+1}}.$$
In this case, we finally obtain
$$\prod_{m=0,m\ne s}^{d-1}|e_s-e_m|=2\cdot 2^{n_{j+1}}\cdot 2^{r-1}\cdot 2^{-n_{j+1}}=2^r.$$
We now treat the case $r\geq1$ and $0\leq s\leq 2^{n_0}-1$. The proof is the same as above. Indeed, thanks to Lemma \ref{lem:simplify.factor} and (\ref{eqn:factor.to.trigon}), we can write
\begin{eqnarray}
\prod_{m=0,m\ne s}^{d-1}|e_s-e_m|
&=&\prod_{m=d_{-1},m\ne s}^{d_0-1}|e_s-e_m| \prod_{k=0}^{r-1}\prod_{m=d_{k}}^{d_{k+1}-1}|e_s-e_m|\\
&=&2^{n_0}2^r \prod_{k=0}^{r-1}|\sin 2^{n_{k+1}-1}(\theta_s-\varphi_0-\cdots-\varphi_{k})|.\label{eqn:trigo.prod.2}
\end{eqnarray}
Now using Lemma \ref{lem:ineq+} in (\ref{eqn:trigo.prod.2}) we get 
$$\prod_{m=0,m\ne s}^{d-1}|e_s-e_m|\geq 2^{n_0}2^r 2^{-n_0+n_r}|\cos 2^{n_0-1}\theta_s|=2^{r+n_r}\geq 2^r,$$
since $|\cos 2^{n_0-1}\theta_s|=1$ for $0\leq s\leq 2^{n_0}-1$. The last case is when $r=1$ and $2^{n_0}\leq s\leq 2^{n_0}+2^{n_1}-1$. The proof is simple and we omit it. 
 \end{proof}
\begin{remark} Take $d=2^{n_0}+1$ and $s=2^{n_0}$ then $\prod_{m=0}^{2^{n_0}-1}|e_{2^{n_0}}-e_m|=|e_{2^{n_0}}^{2^{n_0}}-1|=|-1-1|=2$. Thus the conclusion of Lemma \ref{pro:produc-leja} is optimal in some cases.
\end{remark}

\section{Proof of Theorems \ref{thm.second.lebesgue.type} and \ref{thm.third.lebesgue.type} }
We continue to use the notation introduced in the previous two sections. Let $E=(e_n:n\in\NN)$ be a Leja sequence for $D$ and $\theta_n=\textup{arg}\,e_n$, $n\geq 0$. We use the formula for the polynomials $h_{st}$ and $h'_{st}$ given in (\ref{eqn:hst.new}) and (\ref{eqn:deri.hst.new}). We start with the following simple observations    
$$e_u-e_v=2i\sin\frac{\theta_u-\theta_v}{2}e^{\frac{\theta_u+\theta_v}{2}i}\quad \textup{and}\quad (e_u-e_v)^{\bot}=-2\sin\frac{\theta_u-\theta_v}{2}e^{\frac{\theta_u+\theta_v}{2}i}\quad \text{for}\quad u\ne v.$$
Set $\alpha_{uv}:=e^{\frac{\theta_u+\theta_v}{2}i}$ and $\beta_{uv}:=-2\sin\frac{\theta_u-\theta_v}{2}$. Then we immediately see that
\begin{equation}\label{eqn:simple.equality}
(e_u-e_v)^{\bot}=\alpha_{uv}\beta_{uv}\quad\text{and}\quad |e_u-e_v|=|(e_u-e_v)^{\bot}|=|\beta_{uv}|\quad \text{for}\quad u\ne v.
\end{equation}
\begin{lemma}\label{lem:deno-estimate}
Let $E=(e_n:n\in \NN)$ be a Leja sequence for $D$ and $d\geq 2$, $d=2^{n_0}+2^{n_1}+\cdots+2^{n_r}$ with $n_0>n_1>\cdots>n_r\geq 0$. Then
\begin{equation}\label{eqn:deno-estimate}
|h'_{st}\big(\langle (e_s-e_t)^{\bot},e_s\rangle\big)|\geq \frac{2^{2r}|\beta_{st}|^{d-4}}{2^{d-2}},\quad 0\leq s<t\leq d-1.
\end{equation}
\end{lemma}
\begin{proof}
Thanks to (\ref{eqn:simple.equality}) we may write
\begin{eqnarray}
|h'_{st}\big(\langle (e_s-e_t)^{\bot},e_s\rangle\big)| 
&=&\prod_{m=0,m\ne s,t}^{d-1} |\langle (e_s-e_t)^{\bot},(e_s-e_m)\rangle|\quad\,\,(\textup{ see  (\ref{eqn:deri.hst.new}) })\notag\\ 
&=&|\beta_{st}|^{d-2}  \prod_{m=0,m\ne s,t}^{d-1} \Big|\left\langle \alpha_{st},\frac{e_s-e_m}{-\beta_{sm}}\right\rangle\Big|  \prod_{m=0,m\ne s,t}^{d-1} |e_s-e_m|\notag \\
&=&|\beta_{st}|^{d-3}  \prod_{m=0,m\ne s,t}^{d-1} \Big|\left\langle \alpha_{st},\frac{e_s-e_m}{-\beta_{sm}}\right\rangle\Big|  \prod_{m=0,m\ne s}^{d-1} |e_s-e_m|.\label{eqn:denom-rewrite}
\end{eqnarray}
Since $\frac{e_s-e_m}{-\beta_{sm}}=ie^{\frac{i(\theta_s+\theta_m)}{2}}=e^{\frac{i(\pi+\theta_s+\theta_m)}{2}}$ and $\alpha_{st}=e^{\frac{i(\theta_s+\theta_t)}{2}}$, we have 
\begin{equation}\label{eqn:outer-product}
 \Big|\left\langle \alpha_{st},\frac{e_s-e_m}{-\beta_{sm}}\right\rangle\Big|=\Big|\cos (\frac{\pi+\theta_s+\theta_m}{2}-\frac{\theta_s+\theta_t}{2})\Big|=\Big|\sin\frac{\theta_t-\theta_m}{2}\Big|=\frac{|e_t-e_m|}{2}.
\end{equation}
Combining (\ref{eqn:denom-rewrite}) and (\ref{eqn:outer-product}) we obtain 
\begin{eqnarray*}
 |h'_{st}\big(\langle (e_s-e_t)^{\bot},e_s\rangle\big)| 
&=&|\beta_{st}|^{d-3}  \prod_{m=0,m\ne s,t}^{d-1}\frac{|e_t-e_m|}{2}\prod_{m=0,m\ne s}^{d-1} |e_s-e_m|\\
&=&\frac{|\beta_{st}|^{d-4}}{2^{d-2}}\prod_{m=0,m\ne t}^{d-1}|e_t-e_m|\prod_{m=0,m\ne s}^{d-1} |e_s-e_m|\\
&\geq&\frac{2^{2r}|\beta_{st}|^{d-4}}{2^{d-2}},
\end{eqnarray*}
where we use Lemma \ref{pro:produc-leja} in the third line. 
\end{proof}

\begin{lemma}\label{lem:nume-estimate-kerg}
Under the same assumptions of Lemma \ref{lem:deno-estimate}, we have
\begin{equation}\label{eqn:nume-estimate}
|h_{st}\big(\langle (e_s-e_t)^{\bot},x\rangle\big)|\leq \frac{2^{2r+1}d|\beta_{st}|^{d-2}}{2^{d-2}},\quad 0\leq s<t\leq d-1, \quad\|x\|\leq 1.
\end{equation}
\end{lemma}
\begin{proof} 
 Let us set $\langle\alpha_{st},x\rangle=\cos\varphi$. Since $(e_s-e_t)^{\bot}=\alpha_{st}\beta_{st}$ and $\langle\alpha_{st},e_m\rangle=\cos(\theta_m-\frac{\theta_s+\theta_t}{2})$, in view of  (\ref{eqn:hst.new}), we get 
\begin{eqnarray}
|h_{st}\big(\langle (e_s-e_t)^{\bot},x\rangle\big)|
&=&|\beta_{st}|^{d-1}\prod_{m=0,m\ne s}^{d-1} |\langle \alpha_{st},(x-e_m)\rangle|\notag\\
&=&|\beta_{st}|^{d-1}\prod_{m=0,m\ne s}^{d-1}|\cos\varphi-\cos(\theta_m-\frac{\theta_s+\theta_t}{2})|.\label{eqn:numer1}
\end{eqnarray}
There exists a unique $-1\leq j\leq r-1$ such that $d_j\leq s\leq d_{j+1}-1$, where the $d_i$'s are defined in (\ref{eqn:set.dj}). We can write the trigonometric expression in the right hand side of (\ref{eqn:numer1}) as follows
\begin{equation}\label{eqn:group.factors}
\prod_{k=-1,k\ne j}^{r-1}\prod_{m=d_k}^{d_{k+1}-1}|\cos\varphi-\cos(\theta_m-\frac{\theta_s+\theta_t}{2})| \cdot \prod_{m=d_j,m\ne s}^{d_{j+1}-1}|\cos\varphi-\cos(\theta_m-\frac{\theta_s+\theta_t}{2})|.
\end{equation}
Thanks to Lemmas \ref{lem:simplify.factor.triog} and \ref{lem:one.argument.missing}, the first factor and the second factor in (\ref{eqn:group.factors}) are dominated respectively by 
$$\prod_{k=-1,k\ne j}^{r-1}\frac{2}{2^{2^{n_{k+1}}-1}}\quad\text{and}\quad\frac{2\cdot 2^{n_{j+1}}}{2^{2^{n_j+1}-1}|\sin\frac{\theta_s-\theta_t}{2}|}.$$
Combining these estimates with (\ref{eqn:numer1}) we obtain
$$|h_{st}\big(\langle (e_s-e_t)^{\bot},x\rangle\big)|\leq |\beta_{st}|^{d-1}\cdot \prod_{k=-1}^{r-1}\frac{2}{2^{2^{n_{k+1}}-1}}\cdot\frac{2^{n_{j+1}}}{|\sin\frac{\theta_s-\theta_t}{2}|}\leq  \frac{2^{2r+1}d|\beta_{st}|^{d-2}}{2^{d-2}},$$
here we use the facts that $|\beta_{st}|=2|\sin\frac{\theta_s-\theta_t}{2}|$ and $2^{n_{j+1}}\leq d$. 
\end{proof}
\begin{lemma}\label{lem:nume-estimate-hakop}
Under the same assumptions of Lemma \ref{lem:deno-estimate}, we have
\begin{equation}\label{eqn:nume-estimate}
|h'_{st}\big(\langle (e_s-e_t)^{\bot},x\rangle\big)|\leq  \frac{2^{2r+1}d^3|\beta_{st}|^{d-3}}{2^{d-2}},\quad  0\leq s<t\leq d-1, \,\,\|x\|\leq 1.
\end{equation}
\end{lemma}
\begin{proof}
 Let $\ell: \RR^2\to \RR$ be the linear form defined by  $\ell(x):=\langle (e_s-e_t)^{\bot},x\rangle$. Since $|(e_s-e_t)^{\bot}|=|\beta_{st}|$, we have  $\ell(D)=[-|\beta_{st}|,|\beta_{st}|]$. It follows that 
$$\sup_{x\in D}|h_{st}\big(\langle (e_s-e_t)^{\bot},x\rangle\big)|=\|h_{st}\|_{ [-|\beta_{st}|,|\beta_{st}|]}\quad\text{and}\quad\sup_{x\in D}|h'_{st}\big(\langle (e_s-e_t)^{\bot},x\rangle\big)|=\|h'_{st}\|_{ [-|\beta_{st}|,|\beta_{st}|]}.$$ 
Since $\deg h_{st}=d-1$, classical Markov's inequality gives   
$$\|h'_{st}\|_{[-|\beta_{st}|,|\beta_{st}|]}\leq (1/|\beta_{st}|)(d-1)^2\|h_{st}\|_{ [-|\beta_{st}|,|\beta_{st}|]}.$$ 
Hence Lemma \ref{lem:nume-estimate-kerg} yields
$$\sup_{x\in D}|h'_{st}\big(\langle (e_s-e_t)^{\bot},x\rangle\big)| \leq \frac{(d-1)^2}{|\beta_{st}|} \sup_{x\in D}|h_{st}\big(\langle (e_s-e_t)^{\bot},x\rangle\big)|  \leq \frac{2^{2r+1}d^3|\beta_{st}|^{d-3}}{2^{d-2}},$$
\end{proof}
\begin{proof}[Proof of Theorems \ref{thm.second.lebesgue.type} and \ref{thm.third.lebesgue.type}] 
Lemmas \ref{lem:deno-estimate} and \ref{lem:nume-estimate-kerg} give an upper bound $\frac{2^{2r+1}d|\beta_{st}|^{d-2}}{2^{d-2}}$ for the numerator of $P_{st}(x)$ and a lower bound $\frac{2^{2r}|\beta_{st}|^{d-2}}{2^{d-2}}$ for the denominator of $P_{st}(x)$. Thus $|P_{st}(x)|\leq 2d$ for all $x\in D$ and $0\leq s<t\leq d-1$. At the same time, Lemmas \ref{lem:deno-estimate} and \ref{lem:nume-estimate-hakop} follow that $|Q_{st}(x)|\leq 2d^3|\beta_{st}|\leq 4d^3$ for all $x\in D$ and $0\leq s<t\leq d-1$.
 \end{proof}
\subsection*{Acknowledgement} The author would like to thank his PhD adviser Jean-Paul Calvi for suggesting the problem and for his help in preparing this note. The work is supported by a PhD fellowship from the Vietnamese government. 
\bibliographystyle{plain}
\bibliography{bib_convergence_ker_hak}
\end{document}